\documentclass[12pt]{article}

\usepackage{xcolor}

\usepackage{extarrows}
\usepackage{hyperref}
\usepackage{subfig}
\usepackage[T1]{fontenc}
\usepackage{mathtools}   
\usepackage{amsmath}
\usepackage{url}
\usepackage{amssymb}
\usepackage{amsfonts}
\usepackage{amsthm}
\usepackage{graphicx}
\usepackage{enumerate}
\usepackage{comment}
\usepackage{geometry}
\usepackage{color}
\usepackage[color,matrix,arrow]{xy}

\theoremstyle{definition}
\newtheorem{Def}{Definition}[section]
\newtheorem{Thm}[Def]{Theorem}

\newtheorem{Lem}[Def]{Lemma}
\newtheorem{Pro}[Def]{Proposition}
\newtheorem{Rem}[Def]{Remark}
\newtheorem{Cor}[Def]{Corollary}
\newtheorem{Exa}[Def]{Examples}

\usepackage{xy}
\input xy
\xyoption{all}

\def\IR{\mathbb{R}}
\def\IC{\mathbb{C}}
\def\K{\mathcal{K}}
\def\e{\varepsilon}
\def\Cl{\textup{Cliff}}
\def\A{\mathcal{A}}
\def\wox{\widehat{\otimes}}

\def\IN{\mathbb{N}}
\def\supp{\textup{supp}}
\def\wod{\widehat{\odot}}
\def\IA{\mathbb{A}_u}
\def\H{\mathcal{H}}

\title{$K$-theory of the maximal and reduced Roe algebras of metric spaces with A-by-CE coarse fibrations
*\footnotetext{ Supported in part by NSFC (No.  11831006, 11771143,  12171156).}}
\author{Liang Guo, \and Zheng Luo, \and Qin Wang,\and Yazhou Zhang}
\date{}

\begin{document}
\maketitle
\begin{abstract}
Let $X$ be a discrete metric space with bounded geometry. In this paper, we show that if $X$ admits an "A-by-CE coarse fibration", then the canonical quotient map $\lambda: C^*_{\max}(X)\to C^*(X)$ from the maximal Roe algebra to the Roe algebra of $X$,  and the canonical quotient map $\lambda: C^*_{u, \max}(X)\to C^*_u(X)$ from the maximal uniform Roe algebra to the uniform Roe algebra of $X$, induce isomorphisms on $K$-theory. A typical example of such a space arises from a sequence of group extensions $\{1\to N_n\to G_n\to Q_n\to 1 \}$ such that the sequence $\{N_n\}$ has Yu's property A, and the sequence $\{Q_n\}$ admits a coarse embedding into Hilbert space. This extends an early result of J. \v{S}pakula and R. Willett \cite{JR2013} to the case of metric spaces which may not admit a coarse embedding into Hilbert space.  Moreover, it implies that the maximal coarse Baum-Connes conjecture holds for a large class of metric spaces which may not admit a fibred coarse embedding into Hilbert space.
\end{abstract}

\section{Introduction}

Amenability plays a key role in many areas of mathematics. For a locally compact group $G$, it is well known that $G$ is amenable if and only if the canonical quotient map $\lambda: C^*_{\max}(G)\to C^*_{\mathrm{red}}(G)$ from the maximal to the reduced group $C^*$-algebra is an isomorphism.
Moreover, the $K$-theory functor gives
rise to a morphism 
$\lambda_*: K_*(C^*_{\max}(G))\to K_*(C^*_{\mathrm{red}}(G))$ of abelian groups.  J. Cuntz \cite{Cuntz} introduced $K$-amenability for discrete groups and, in particular,  showed that
$\lambda_*$ is an isomorphism for some non-amenable discrete groups such as free groups $\mathbb{F}_n$ or
$\mathrm{SL}(2, \mathbb{Z})$. N. Higson and G. Kasparov  \cite{HigsonKasparov2001} further proved that $\lambda_*$ is an isomorphism for all
a-T-menable groups, as part of their celebrated work on the Baum-Connes conjecture.
\par
There is a striking parallel between geometric/analytic properties for groups and that for metric spaces in the setting of coarse geometry. The roles of amenability and a-T-menability for groups are played by Yu's {\em property A} and Gromov's {\em coarse embeddability into Hilbert space} for metric spaces, respectively. The roles of the maximal and reduced group $C^*$-algebras for a group are played by the {\em maximal Roe algebra $C^*_{\max}(X)$} and the {\em Roe algebra $C^*(X)$}, respectively,  or the {\em maximal uniform Roe algebra $C^*_{u, \max}(X)$} and the {\em uniform Roe algebra $C^*_{u}(X)$}, respectively, for a metric space $X$.
\par
Recall that the Roe algebra $C^*(X)$ and uniform Roe algebra $C^*_u(X)$ were introduced by J. Roe \cite{Roe1988,Roe1993}  in his work on index theory on general non-compact complete Riemannian manifolds. As a geometric analogue to the Baum-Connes conjecture for groups, the {\em coarse Baum-Connes conjecture} is an algorithm for the $K$-theory $K_*(C^*(X))$ of the Roe algebra $C^*(X)$ of a metric space $X$ with bounded geometry. G. Yu \cite{Yu2000} proved that the coarse Baum-Connes conjecture holds for all metric spaces with bounded geometry which admit a coarse embedding into Hilbert space. Meanwhile, the uniform Roe algebras $C^*_u(X)$ have found applications in a variety of fields such as $C^*$-algebra theory, index theory, topological dynamics, single operator theory, and mathematical physics  (c.f. \cite{rigidity2021,BragaFarahrigid,STYcoarseBaum-Connes,JRrigidity,GeometrypropertyT,WZnucleardimension}).
On the other hand, the maximal Roe algebra $C^*_{\max}(X)$ for a metric space $X$ with bounded geometry was introduced by G. Gong, Q. Wang and G. Yu \cite{GWY2008} in their work on geometrization of the strong Novikov conjecture for residually finite groups. The maximal uniform Roe algebra $C^*_{u, \max}(X)$ can be defined analogously.
\par
By the universality of the maximal completions,  just as in the group case, there are canonical quotient maps
$$\lambda:C^*_{max}(X)\rightarrow C^*(X) \quad \mbox{ and }\quad \lambda:C^*_{u,max}(X)\rightarrow C_u^*(X), $$
which further induce homomorphisms of abelian groups on the $K$-theory level:
$$\lambda_*: K_*(C^*_{max}(X))\rightarrow K_*(C^*(X)) \quad \mbox{ and }\quad 
\lambda_*: K_*(C^*_{u,max}(X))\rightarrow K_*(C_u^*(X)).$$
\par
J. \v{S}pakula and R. Willett \cite{JR2013} proved that, for a discrete metric space $X$ with bounded geometry,  (1) if $X$ has Yu's property A, then the quotient maps 
$\lambda:C^*_{max}(X)\rightarrow C^*(X)$ and $\lambda:C^*_{u,max}(X)\rightarrow C_u^*(X)$ are isomorphisms; (2) if $X$ admits a coarse embedding into Hilbert space, then the induced maps
$\lambda_*: K_*(C^*_{max}(X))\rightarrow K_*(C^*(X))$ and 
$\lambda_*: K_*C^*_{u,max}(X)\rightarrow K_*C_u^*(X)$  on $K$-theory are isomorphisms. Notice that these results are intimately parallel to the group case! As remarked in  \cite{JR2013}, while the statement (1) is fairly straightforward, the proof of the statement (2) relies heavily on the deep work of G. Yu \cite{Yu2000} on the coarse Baum-Connes conjecture, and that of N. Higson, G. Kasparov and J. Trout \cite{HKT1998} on an infinite dimensional Bott periodicity theorem.
\par
Very recently,  J. Deng, Q. Wang and G. Yu \cite{DWY2021} proved that, for a sequence of group extensions $(1\to N_n\to G_n\to Q_n \to 1)_{n\in \mathbb{N}}$ of finitely generated groups with uniformly finite generating subsets, if the coarse disjoint union of $(N_n)_{n\in \mathbb{N}}$ has Yu's property A and the coarse disjoint union of $(Q_n)_{n\in \mathbb{N}}$ admits a coarse embedding into Hilbert space, then the coarse Baum-Connes conjecture holds for the coarse disjoint union of $(G_n)_{n\in \mathbb{N}}$. Such a sequence of group extensions  $(1\to N_n\to G_n\to Q_n \to 1)_{n\in \mathbb{N}}$ is said to have an {\em A-by-CE structure}. It follows that the coarse Baum-Connes conjecture holds for the relative expanders and group extensions constructed by G. Arzhantseva and R. Tessera \cite{relativeexpander2015,AT2018}, and special box spaces of free groups discovered by T. Delabie and A. Khukhro \cite{boxspace2018}, which do not coarsely embed into Hilbert space, yet do not contain any weakly embedded expander.
\par
In this paper, as motivated by the above developments, we shall extend the result of J. \v{S}pakula and R. Willett \cite{JR2013} on $K$-theory from coarsely embeddable spaces (i.e. {\em CE spaces}) to spaces with an {\em A-by-CE structure}. Such a problem has recently been posed by Alexander Engel in his blog \cite{A2021}. 
\par
To get started, recall that a family $\{X_i\}_{i\in I}$  of discrete metric spaces with bounded geometry is said to have {\em equi-property A} if for any $R>0, \epsilon>0$, there exists $S>0$ and a family of maps $\xi^{(i)}: X_i \rightarrow \ell^2(X_i)$, $i\in I$, such that
 	(1) for all $x,y\in X_i$, $\xi^{(i)}_x(y)\in [0,1]$;
 	(2) for all $x\in X_i$, $\|\xi^{(i)}_x\|=1$;
 	(3) $\sup\{\|\xi^{(i)}_x-\xi^{(i)}_y\|:d(x,y)\leq R, x,y \in X_i\}<\epsilon$;
 	(4) for all $x\in X_i$, $\mbox{supp} (\xi^{(i)}_x)\subset B(x,S)$, $i\in I$. 
Inspired by a condition used by M. Dadarlat and E. Guentner in \cite[Corollary 3.3]{DG2007}, it would be helpful to introduce the following notion of {\em A-by-CE coarse fibration}, so as to cover more general spaces than coarse disjoint unions of sequences of group extensions with an A-by-CE structure.  	 
\par
\begin{Def}\label{A-by-CE coarse fibration}
A discrete metric space with bounded geometry $X$ is said to admit an {\em A-by-CE coarse fibration} if there exist a discrete metric space $Y$ with bounded geometry and a map $p:X\rightarrow Y$ satisfying  
the following conditions:
\begin{enumerate}
\item[(1)] the map $p$ is uniformly expansive, i.e., for any $r>0$ there exists $s>0$ such that
$$d_X(x,x')<r \Longrightarrow d_Y(p(x),p(x'))<s \; ;$$
\item[(2)] for any $R>0$, the family of  subspaces  $\{p^{-1}\left(B_Y (y,R)\right)\}_{y\in Y}$ of $X$  has equi-property A;
\item[(3)] the space $Y$ admits a coarse embedding into Hilbert space.
\end{enumerate}  
\end{Def}
\par
Similarly, one may talk about {\em A-by-A} coarse fibration, {\em CE-by-A} coarse fibration, {\em CE-by-CE} coarse fibration and so on. M. Dadarlat and E. Guentner proved in \cite[Corollary 3.3]{DG2007} that a space $X$ with an {\em A-by-A} coarse fibration has Property A, and a space $X$ with  a {\em CE-by-A} coarse fibration is CE, i.e. coarsely embeddable into Hilbert space. On the other hand, however, a space $X$ with an {\em A-by-CE} coarse fibration can be far from being coarsely embeddable into Hilbert space,  as demonstrated in the examples of relative expanders \cite{relativeexpander2015} and group extensions \cite{AT2018} of G. Arzhantseva and R. Tessera,  and special box spaces free groups \cite{boxspace2018} of T. Delabie and A. Khukhro, mentioned above. 
\par
Our main result in this paper is the following theorem.
\begin{Thm}\label{main result}
{\it Let $X$ be a discrete metric space with bounded geometry. If $X$ admits an A-by-CE coarse fibration, then the homomorphisms on $K$-theory
$$\lambda_*: K_*(C^*_{max}(X))\rightarrow K_*(C^*(X)) \quad \mbox{ and }\quad 
\lambda_*: K_*(C^*_{u,max}(X))\rightarrow K_*(C_u^*(X))$$
are isomorphisms. }
\end{Thm}
\par
Recall \cite{Higherindex} that associated to any metric space $X$ with bounded geometry are two assembly maps $\mu$ and $\mu_{\max}$, which fit into the following commutative diagram:

\[
\xymatrix{
	&  & K_*(C^*_{max}(X)) \ar[d]^{\lambda_*}  \\
\displaystyle\lim_{d\to \infty}	K_*(P_d(X))\ar[urr]^{\mu_{max}}\ar[rr]^{\mu} &  &  K_*(C^*(X)) ,
}
\]
where on the left hand side $\lim_{d\to \infty} K_*(P_d(X))$ is the direct limit of the $K$-homology group of the Rips complex $P_d(X)$ of $X$ at scale $d>0$. The coarse Baum-Connes conjecture states that $\mu$ is an isomorphism, while the {\em maximal coarse Baum-Connes conjecture} states that $\mu_{\max}$ is an isomorphism. 
In \cite{CWY2013},  X. Chen, Q. Wang and G. Yu introduced a notion of {\em fibred coarse embedding into Hilbert space}, which is a far generalization to the notion of coarse embedding into Hilbert space, and proved that the maximal coarse
Baum-Connes conjecture holds for metric spaces with  bounded geometry which admit a fibred coarse embedding into Hilbert space. However, a space $X$ with an {\em A-by-CE coarse fibraton} 
can even be far from being fibred coarsely embeddable into Hilbert space. This is revealed by the first two examples of relative expander of G. Arzhantseva and R. Tessera \cite{relativeexpander2015}, which are box spaces of certain {\em abelian-by-free} group extensions 
$\mathbb{Z}^2\rtimes_{Q} \mathbb{F}_3$ and 
$\mathbb{Z} \wr_Q \mathbb{F}_U$. Since these two groups have relative property (T) with respect to an infinite subset (see   \cite[Section 7, Theorem 7.3]{relativeexpander2015} for more details), 
any box space of them does not admit a fibred coarse embedding into Hilbert space (cf. \cite{CWW2013}).
\par
Nevertheless, the proof of the recent result in \cite{DWY2021} by J. Deng, Q. Wang and G. Yu 
for sequences of group extensions with  A-by-CE structure can be adapted to metric spaces with A-by-CE coarse fibration. More precisely, we have the following result.
\par
\begin{Thm}[J. Deng, Q. Wang and G. Yu \cite{DWY2021}]\label{DWY}
{\it Let $X$ be a discrete metric space with bounded geometry. If $X$ admits an A-by-CE coarse fibration, then the coarse Baum-Connes conjecure holds for $X$. }
\end{Thm}
\par
Combined with the main result of this paper, Theorem \ref{main result}, we deduce the following result. 
\par
\begin{Cor}
{\it The maximal coarse Baum-Connes conjecture holds for all discrete metric spaces with bounded geometry which admits an A-by-CE coarse fibration.  }
\end{Cor}
\par
In particular, the maximal coarse Baum-Connes conjecture holds for the first two examples of relative expanders of G. Arzhantseva and R. Tessera \cite{relativeexpander2015} which do not admit a fibred coarse embedding into Hilbert space.
Of course, this fact has already been observed in \cite{relativeexpander2015} by appealing certian results on the Baum-Connes conjecture (cf. \cite{HoyonoYu09}). Corollary 1.4 serves as an alternative proof to this fact.
\par
The proof of Theorem \ref{main result} relies on the work of J. Deng, Q. Wang and G. Yu \cite{DWY2021}. We will mainly focus on the proof for the (maximal and reduced) {\em uniform} Roe algebras, i.e. the map $\lambda_*: K_*(C^*_{u,max}(X))\rightarrow K_*(C_u^*(X))$ is an isomorphism. Due to the participation of compact coefficients, the (maximal and reduced) Roe algebras are more stable on $K$-theory, so that the proof for the case of Roe algebras is almost the same as that for the uniform algebras, and a little more simpler.
\par
This paper is organized as follows. In Section 2, we recall the preliminary notions of the Roe algebra, the uniform Roe algebra and the uniform algebra, and their maximal counterparts. In Section 3, we introduce a twisted uniform Roe algebra for a space with an A-by-CE coarse fibration,  and prove the maximal and reduced versions of this algebra are isomorphic at the $C^*$-algebra level. In Section 4, we use a variant of coarse Bott and Dirac asymptotic morphisms to prove our main result. 
\par 
As an additional remark, note that G. Arzhantseva and R. Tessera  \cite{relativeexpander2015} showed that no expanders can be weakly embedded in a sequence of group extensions with {\em CE-by-CE structure}.  This might support an affirmative answer to the following question:
\par
{\bf Problem:} Is it true that $\lambda_*: K_*(C^*_{max}(X))\rightarrow K_*(C^*(X))$ and $\lambda_*: K_*(C^*_{u,max}(X))\rightarrow K_*(C_u^*(X))$ are isomorphisms for all discrete metric spaces $X$ with bounded geometry which admit a {\em CE-by-CE coarse fibration}?

\section{Preliminaries}

In this section, we will briefly recall the Roe algebras, the uniform Roe algebras and the uniform algebras for a discrete metric space with bounded geometry, and their maximal counterparts. We will also recall several relevant notions in coarse geometry.

Let $X$ be a discrete  metric space with bounded geometry. Recall that a discrete metric space $X$ is said to have bounded geometry if for any $r> 0$ there exists
$N_r >0$ such that any ball of radius $r$ in $X$ contains at most $N_r$ points. A discrete metric space $X$ is said to be uniformly discrete if there exists $\delta>0$ such that for any two distinct points $x,y\in X$, we have $d(x,y)>\delta$.  
\par
Let $T=(T(x,y))$ be an $X$-by-$X$ matrix, where each $T(x, y)$ is an element in some algebra. $T$ is said to have \emph{finite propagation} if there exists $s>0$ such that $T(x,y)=0$ whenever $d(x,y)\geq s$. Fix a separable infinite dimensional Hilbert space $H$, and let $\mathcal{K}:=\mathcal{K}(H)$ denote the algebra of compact operators on $H$. 

\begin{Def}
The {\em algebraic uniform Roe algebra} of $X$, denoted by $\mathbb{C}_u[X]$, is the algebra of all $X$-by-$X$ matrices $T=(T(x,y))_{x,y\in X}$ with $T(x, y)\in \mathbb{C}$, the complex numbers, such that $T$ has finite propagation and $\sup_{x,y\in X}{|T(x,y)|}<\infty$. For fixed $s\geq 0$, we denote by $\mathbb{C}_u^s[X]$ the subspace consisting of those $T$ such that $\mbox{propagation}(T)\leq s$.

The {\em algebraic Roe algebra} of $X$, denoted by $\IC[X,\K]$, is the set of finite propagation $X$-by-$X$ indexed matrices $(T(x,y))_{x,y\in X}$ with entries in $\K$ satisfying $$\sup\limits_{x,y\in X}\|T(x,y)\|<\infty.$$ For fixed $s>0$, $\IC^s[X,\K]$ is defined analogously to the uniform case.
\end{Def}

Both $\IC_u[X]$ and $\IC[X,\K]$ are equipped with $*$-algebra structures by using the usual matrix operations, i.e., for any $T_1,T_2\in\IC_u[X]$ (respectively, $T_1,T_2\in\IC[X,\K]$), a product structure on $\IC_u[X]$ (resp. $\IC[X,\K]$) is defined by
$$(T_1T_2)(x,y)=\sum_{z\in X}T_1(x,z)T_2(z,y),$$
and the adjoint $T_1^*$ of $T_1$ is defined by
$$T_1^*(x,y)=(T_1(y,x))^*.$$
Note that $\IC_u[X]$ admits a natural $*$-representation by 'matrix multiplication' on $\ell^2(X)$, and similarly $\IC[X,\K]$ admits a natural $*$-representation on $\ell^2(X,H)$.

\begin{Def}
The {\em uniform Roe algebra} of $X$, denoted $C^*_u(X)$, is the completion of $\IC_u[X]$ for its natural representation on $\ell^2(X)$. The {\em Roe algebra} of $X$, denoted $C^*(X)$, is the completion of $\IC[X,\K]$ for its natural representation on $\ell^2(X,H)$.
\end{Def}

The following lemma is proved in \cite[Section 3]{GWY2008}, which shows that the maximal completion of $\mathbb{C}_u[X]$ (resp. $\IC[X,\K]$) is well-defined.

\begin{Lem}[\cite{GWY2008}]\label{maxnorm}
{\it
For all $s>0$ there exists a constant $C_s$ such that for all $T\in\IC_u^s[X]$ (resp. $T\in\IC[X,\K]$) and any $*$-representation $\pi:\mathbb{C}_u^s[X]\rightarrow B(H)$ (resp. $\pi:\IC[X,\K]\to\mathcal{B}(H)$), one has that
$$\|\pi(T)\|_{B(H)}\leq C_s \sup\limits_{x,y\in X}\|T(x,y)\|.$$
}
\end{Lem}

\begin{Def}\label{maxuniRoe}
The {\em maximal uniform Roe algebra} of $X$, denoted $C_{u,max}^*(X)$, is the completion of $\mathbb{C}_u[X]$ for the norm
	$$\|T\|=\sup\{\|\pi(T)\|_{B(H)}\mid \pi:\mathbb{C}_u[X]\rightarrow B(H) \text{ a $*$-representation }\}.$$
We can also define the {\em maximal Roe algebra}, denoted $C^*_{max}(X)$, analogously.
\end{Def}

The Roe algebra $C^*(X)$ and the uniform Roe algebra $C^*_u(X)$ are occasionally referred to as the {\em reduced Roe algebra} and the {\em reduced uniform Roe algebra}, respectively, so as to emphasize the distinctions from the maximal ones.

Next, we recall another relevant Roe-type $C^*$-algebra associated to a discrete metric space with bounded geometry, which was introduced in \cite{JR2013}.

\begin{Def}
Define $U\mathbb{C}[X]$ to be the $*$-algebra consisting of finite propagation $X$-by-$X$ matrices $(T(x,y))_{x,y\in X}$ with entries in $\mathcal{K}$ such that $$\sup\limits_{x,y\in X}\|T(x,y)\|<\infty  \text{ and } \sup\limits_{x,y\in X}\mbox{rank} (T(x,y))<\infty$$
where $\mbox{rank} ( T(x,y))$ is the rank of the compact operator $T(x,y)$.

$U\mathbb{C}[X]$ has a natural representation on $\ell^2(X)\otimes H$ via matrix multiplication. Define $UC^*(X)$ to be the operator norm closure of $U\mathbb{C}[X]$ on $B(\ell^2(X)\otimes H)$.
\end{Def}
Analogous to the maximal uniform Roe algebra, one can define $UC_{max}^*(X)$ to be the closure of $U\mathbb{C}[X]$ under the maximal norm by Lemma \ref{maxnorm}.

\begin{Rem}\label{Morita equivalence}
	J. \v{S}pakula and R. Willett \textup{\cite[Proposition 4.7]{JR2013}} constructed a pre-Hilbert $U\mathbb{C}[X]$-$\mathbb{C}_u[X]$-bimodule $E_{alg}$ consisting of all finite propagation $X$-by-$X$ matrices with uniformly bounded entries in $H$, and then established the Morita equivalence between the (maximal) uniform algebra and the (maximal) uniform Roe algebra via different completion of pre-Hilbert module $E_{alg}$. By a result of R.Exel in \cite{Exel1993}, a Morita equivalence bimodule $E$ induces an isomorphism $E_*$ on $K$-theory
	$$E_*: K_*(UC^*(X))\to K_*(C^*_{u}(X)),$$
	and the maximal counterpart
	$$E_*: K_*(UC^*_{\max}(X))\to K_*(C^*_{u, \max}(X)),$$
	without any assumptions on separability or $\sigma$-unitality.
		J. \v{S}pakula and R. Willett  \textup{\cite[Proposition 4.8]{JR2013}} also gave a description of the inverse of $E_*$ as follows. Fix a unit vector $\eta_0\in H$. Let $P_0\in\mathcal{K}(H)$ be a rank-one projection onto the subspace of $H$ spanned by $\eta_0$. Let $P:=\mathrm{diag}(P_x)\in \ell^\infty (X, \mathcal{K}(H))$ be the diagonal operator with $P_x=P_0$ at  every $(x,x)$-entry.
	Define an injective homomorphism $i_P:\mathbb{C}_u[X]\rightarrow U\mathbb{C}[X]$ as $(i_P(T))=P(T\otimes Id_H)$, for all $T=(T(x,y))\in \mathbb{C}_u[X]$. The $*$-homomorphisms $i_P$ induce an isomorphism on $K$-theory which is the inverse of $E_*$.
	
\end{Rem}

Finally, we recall the notion of coarse equivalence for families of metric spaces, and illustrate it by an example of a group extension.

Let $(X,d_X)$ and $(Y,d_Y)$ be proper metric spaces. A map $f:X\to Y$ is said to be a coarse embedding if there exist two non-decreasing functions $\rho_1,\rho_2:[0,\infty)\to [0,\infty)$ such that
\begin{enumerate}
	\item [(1)]$\lim\limits_{t\rightarrow \infty}\rho_1(t)=\infty;$
	\item [(2)]$\rho_1(d_X(x,y))\leq d_Y(f(x),f(y))\leq \rho_2(d_X(x,y))$ for all $x,y\in X$.
\end{enumerate}
A metric space $X$ is said to be {\em coarsely equivalent} to another metric space $Y$, if there exists a coarse
embedding $f:X\to Y$ such that $Y$ is equal to the $C$-neighborhood of the image $f(X)$ for some $C> 0$. 

Let $\{(X_i,d_{X_i})\}_{i\in I}$  and $\{(Y_i,d_{Y_i})\}_{i\in I}$ be families of metric spaces with uniform bounded geometry in the sense that for each $R>0$ there exists $M >0$ such that any ball of radius $R$ in each $X_i$ (resp. $Y_i$) contains at most $M$ points, which does not depend on $i$. 
\par
We say that $\{X_i\}_{i\in I}$ is {\em equi-coarsely equivalent} or {\em uniformly coarsely equivalent} to $\{Y_i\}_{i\in I}$ if there exists a sequence of maps $\{f_i:X_i\to Y_i\}_{i\in I}$, a constant $C>0$ and two  non-decreasing functions $\rho_1,\rho_2:[0,\infty)\to[0,\infty)$,  and such that
\begin{itemize}
	\item[(1)]$\lim\limits_{t\rightarrow \infty}\rho_1(t)=\infty;$
	\item[(2)]$\rho_1(d_{X_i}(x,y))\leq d_{Y_i}(f_i(x),f_i(y))\leq \rho_2(d_{X_i}(x_,y))$, for all $x,y\in X_i$ and $i\in I$.
	\item [(3)] $Y_i$ is equal to the $C$-neighborhood of the image $f_i(X_i)$ for each $i\in I$.
\end{itemize}

	


\begin{Exa}\label{group}
	Let $G$ be a countable discrete group, endowed with a left invariant metric via a proper length function on $G$. Assume that $N\lhd G$ is a normal subgroup of $G$ which is endowed with the induced metric, and $G/N$ is the quotient group which is endowed with the quotient metric. Let $\pi:G\to G/N$ be the canonical quotient map. It is a contractive map.
	
	For each $q\in G/N$, choose $g_q\in G$ such that $\pi(g_q)=q$, the fibre space of $q$ is exactly the coset $\pi^{-1}(q)=g_qN=Ng_q$. The left invariant metric of $G$ implies that the fibre $\{\pi^{-1}(q)\}$ is isometric to $\{\pi^{-1}(r)\}$ for any $q,r\in G/N$ by
	$$g_qN\to g_rN\mbox{,}\quad g_qn\mapsto g_rn.$$
	
	For any $R>0$, if $p,q\in G/N$ satisfy $d_{G/N}(p,q)<R$, then for any $x\in \pi^{-1}(p)$, there exists $y\in\pi^{-1}(q)$ such that $d_G(x,y)<R$. The existence of $y$ is guaranteed by the definition of the quotient metric. As a result, there exists $C\geq R$ such that $\pi^{-1}(q)$ is a $C$-net in $\pi^{-1}(B_{G/N}(q,R))$ for any $q\in G/N$. Hence the sequence of metric spaces $\{\pi^{-1}(B_{G/N}(q,R))\}_{q\in G/N}$ is uniformly coarse equivalent to $\{\pi^{-1}(q)\}_{q\in G/N}$. If $N$ has property A and $Q$ is coarsely embeddable into Hilbert space, then the map $\pi:G\to G/N$ provides $G$ with an A-by-CE coarse fibration. 
\end{Exa}

\section{The twisted (uniform) Roe algebra}\label{Section 3}

Let $X$ be a discrete metric with bounded geometry such that a map $p:X\to Y$ endows $X$ with an A-by-CE coarse fibration. In this section, we will construct the (maxima and reduced) twisted uniform Roe algebras and Roe algebras for $X$, with coefficients coming from the
coarse embedding of $Y$ into a Hilbert space $H$. The constructions of these algebras are  variants of those twisted algebras originally introduced in \cite{Yu2000}. We then prove that the canonical quotient maps from the
maximal twisted (uniform) Roe algebras to the reduced twisted (uniform) Roe algebras are isomorphisms.

Throughout this section, assume that $X$ and $Y$ are bounded geometry discrete  metric spaces with metric $d_X$ and $d_Y$, respectively. Let $H$ be a separable and infinite-dimensional Hilbert space. Moreover, we assume that $p:X \to Y$ satisfies the conditions in Theorem \ref{main result} and $f:Y\to H$ is a coarse embedding.

To get started, let us first recall a $C^*$-algebra associated with an infinite-dimensional Euclidean space introduced by N. Higson, G. Kasparov and J. Trout in \cite{HKT1998}. Let $V$ be a countably infinite-dimensional Euclidean space. Denote by $V_a, V_b$ the finite dimensional affine subspaces of $V$. Denoted by $V_a^0$ the finite dimensional linear subspace of $V$ consisting of differences of elements in $V_a$.  Let $\Cl(V_a^0)$ be the complexified Clifford algebra on $V_a^0$ and $\mathcal{C}(V_a)$ the graded $C^*$-algebra of continuous functions from $V_a$ to $\text{Cliff}(V_a^0)$ vanishing at infinity. Let $\mathcal{S}=C_0(\mathbb{R})$, graded according to even and odd functions. Define the graded tensor product
$$\mathcal{A}(V_a)=\mathcal{S}\wox\mathcal{C}(V_a).$$
If $V_a \subseteq V_b$, we have a decomposition $V_b=V_{ba}^0+ V_a$, where $V_{ba}^0$ is the orthogonal complement of $V_a^0$ in $V_b^0$. For each $v_b \in V_b$, we have a unique decomposition $v_b=v_{ba}+v_a$, where $v_{ba} \in V_{ba}^0$ and $v_a\in V_a$. Every function $h$ on $V_a$ can be extended to a function $\tilde{h}$ on $V_b$ by the formula $\tilde{h}(v_{ba}+v_a)=h(v_a)$.

\begin{Def}[\cite{HKT1998}] \leavevmode
    (1) If $V_a \subseteq V_b$, we define $C_{ba}$ to be the Clifford algebra-valued function $V_b \rightarrow \Cl(V_b^0)$, $v_b \mapsto v_{ba}\in  V_{ba}^0 \subset \Cl(V_b^0)$. Let $X$ be the function multiplication by $x$ on $\mathbb{R}$, considered as a degree one and unbounded multiplier of $\mathcal{S}$.  Define a homomorphism 
    $$\beta_{ba}: \mathcal{A}(V_a) \rightarrow \mathcal{A}(V_b)$$ 
    by
	$$\beta_{ba}(g\wox h)=g(X\wox 1+1\wox C_{ba})(1\wox\tilde{h})$$
for all $g\in \mathcal{S}$ and $h\in\mathcal{A}(V_a)$, and $g(X\wox 1+1\wox C_{ba})$ is defined to be the functional calculus of the unbounded, essentially self-adjoint multiplier $X\wox 1+1\wox C_{ba}$ by $g$.
\par
(2) We define a $C^*$-algebra $\mathcal{A}(V)$ by 
$$\mathcal{A}(V)=\lim\limits_{\longrightarrow}\mathcal{A}(V_a),$$
where the direct limit is taken over the directed set of all finite-dimensional affine subspaces $V_a\subset V$, using the homomorphism $\beta_{ba}$ in \emph{(1)}.
\end{Def}

\begin{Rem}\label{composition}
If $V_a \subset V_b \subset V_c$, then we have $\beta_{cb} \circ \beta_{ba}=\beta_{ca}$. Therefore, the above homomorphisms give a directed system $(\mathcal{A}(V_a))$ as $V_a$ ranges over finite dimensional affine subspaces of $V$.
\end{Rem}

The space $\mathbb{R}_+ \times H$ is equipped with a topology under which $\mathbb{R}_+ \times H$ is a locally compact topological space in such a way that a net $(t_i,v_i)$ in $\mathbb{R}_+ \times H$  converges to a point $(t,v)\in \mathbb{R}_+ \times H$ if and only if
\begin{itemize}
\item[(1)] $t_i^2+\|v_i\|^2 \rightarrow t^2+\|v\|^2$, as $i \rightarrow \infty;$
\item[(2)] $\langle v_i,u\rangle \rightarrow \langle v,u\rangle $ for any $u\in H$, as $i \rightarrow \infty.$
\end{itemize}
Note that for each $v\in H$ and each $r > 0$, $B(v,r)=\{(t,w)\in \mathbb{R}_+ \times H:t^2+\|v-w\|^2 <r^2\}$ is an open subset of $\mathbb{R}_+ \times H$. For finite dimensional subspaces $V_a \subseteq V_b \subset V$, since $\beta_{ba}$ takes $C_0(\mathbb{R}_+ \times V_a)$ into $C_0(\mathbb{R}_+ \times V_b)$, the $C^*$-algebra $\lim\limits_{\longrightarrow}C_0(\mathbb{R}_+ \times V_a)$ is $*$-isomorphic to $C_0(\mathbb{R}_+ \times H)$.

\begin{Def}
The support of an element $a\in \mathcal{A}(V)$ is the complement of all $(t,v)\in \mathbb{R}_+ \times H$ such that there exists $g\in C_0(\mathbb{R}_+ \times H)$ with $g(t,v) \neq 0 $ and $g\cdot a=0$.
\end{Def}

For any $x\in X$ and $n\in\IN$, define a finite-dimensional Euclidean subspace of $H$ as
$$W_n(p(x))=\mbox{span}\{f(y) \mid y\in Y\mbox{ and }d_Y(p(x),y)\leq n^2\}.$$
where $f:Y\to H$ is the coarse embedding. The bounded geometry property of $X$ implies that $W_n(p(x))$ is finite dimensional. Let $V=\cup_{n\in \mathbb{N}}W_n(p(x))$. Without loss of generality we assume that $V$ is dense in $H$.

\begin{Rem}\leavevmode
(1) Since $Y$ has bounded geometry, for each $n\in\mathbb{N}$, there exists $R>0$ such that $\dim(W_n(p(x)))\leq R$ for all $x\in X$.
(2)by the condition (1) of Definition \ref{A-by-CE coarse fibration}, for any $r>0$, there exists $s>0$ such that $d_X(p(x),p(y))<s$ for all $x,y\in Y$ with $d_Y(x,y)<r$, it follows that there exists $N>0$ such that $W_n(p(x))\subset W_{n+1}(p(y))$ for all $n \geq N$, $x,y \in X$ satisfying $d_X(x,y)\leq r.$
\end{Rem}

\begin{Def}\label{algebraic twisted uniform Roe algebra}
The {\em algebraic twisted uniform Roe algebra}, denoted $\mathbb{C}_u[X,\A]$,  consists of all functions $T$ on $Y \times Y\rightarrow \mathcal{A}(H)$ such that
\begin{itemize}
\item [(1)]there exists an integer $N$ such that $$T(x,y)\in (\beta_N(p(x)))(\mathcal{A}(W_N(p(x))))\subseteq \mathcal{A}(V)$$ for all $x,y\in X$, where $\beta_N(p(x)): \mathcal{A}(W_N(p(x)))\rightarrow \mathcal{A}(V)$ is the $*$-homomorphism associated to the inclusion of $W_N(p(x))$ into $V$;
\item [(2)]there exists $M>0$ such that $\|T(x,y)\|\leq M$ for all $x,y \in X;$
\item [(3)]there exists $r_1>0$ such that if $d_X(x,y)>r_1$, then $T(x,y)=0;$
\item [(4)]there exists $r_2>0$ such that $\mbox{supp}(T(x,y))\subseteq B(f(p (x)),r_2)$ for all $x,y\in X$ where $$B(f(p(x)),r_2)=\left\{(s,h)\in \mathbb{R}_+\times H\mid s^2+\|h-f(p(x))\|^2 < r_2^2\right\};$$
\item [(5)]there exists $c>0$ such that $D_Z(T_1(x,y))\in \mathcal{A}(W_N(p(x)))$ and $\|D_Z(T_1(x,y))\|\leq c$ for all $x,y \in X$ and $Z=(s,h)\in\mathbb{R}\times W_N(p(x))$ satisfying $\|Z\|=\sqrt{s^2+\|h\|^2}\leq 1$, where $\beta_N(x)(T_1(x,y))=T(x,y)$ and $D_Z(T_1(x,y))$ is the derivate of the function $T_1(x,y):\mathbb{R}\times W_N(p(x))\to\Cl(W_N(p(x)))$ in the direction of $Z$.
\item[(6)] for all $x,y\in X$, $T(x,y)$ can be written as the image under $\beta_N(p(x))$ of a finite linear combination of elementary tensors from $\A(W_N(p(x)))\cong \mathcal{S}\wox\mathcal{C}(W_N(p(x)))$.
	
\end{itemize}
\end{Def}

The algebraic twisted uniform Roe algebra $\IC_u[X,\A]$ is made into a $*$-algebra via matrix operations, together with the $*$-operations on $\A(V)$, i.e., for any $T,S\in\IC_u[X,\A]$, a product structure on $\IC_u[X,\A]$ is defined by
$$(TS)(x,y)=\sum_{z\in X}T(x,z)S(z,y)$$
and the adjoint $T^*$ of $T$ is defined by
$$T^*(x,y)=(T(y,x))^*.$$

Let 
$$E=\left\{\sum_{x\in X} a_x[x]\,\Big|\, a_x\in \mathcal{A}(V), \sum_{x\in X}a_x^*a_x \text{ converges in norm } \right\}.$$
The $\A(V)$-valued inner product and $\A(V)$-action are given respectively by
$$\left<\sum_{x\in X} a_x[x],\sum_{x\in X} b_x[x]\right>=\sum_{x\in X} a_x^*b_x\quad\mbox{and}\quad\left(\sum_{x\in X} a_x[x]\right)a=\sum_{x\in X} a_xa[x].$$
Then $E$ becomes a Hilbert $\A(V)$-module, and the action of $\mathbb{C}_u[X,\mathcal{A}]$ on $E$ is given by
$$T\left(\sum\limits_{x\in X} a_x[x]\right)=\sum\limits_{y\in X}\left(\sum_{x\in X} T(y,x)a_x[x]\right)[y],$$
where $T\in \mathbb{C}_u[X,\mathcal{A}]$ and $\sum\limits_{x\in X} a_x[x] \in E.$

\begin{Def}\label{twisted uniform Roe algebra}
The {\em twisted uniform Roe algebra} $C_u^*(X,\mathcal{A})$ is defined to be the operator norm closure of $\mathbb{C}_u[X,\mathcal{A}]$ in $\mathcal{B}(E)$, where $\mathcal{B}(E)$ is the $C^*$-algebra of all adjointable module homomorphisms from $E$ to $E$.
\end{Def}

One can define the universal norm on $\IC_u[X,\A]$ to be the supremum over all norms coming from $*$-representations to the bounded operators on a Hilbert space. It is well-defined by an analogue of Lemma \ref{maxnorm}.

\begin{Def}
	The {\em maximal twisted uniform Roe algebra} $C^*_{u,max}(X,\mathcal{A})$, is defined to be the completion of  $\mathbb{C}_u[X,\mathcal{A}]$ in the norm
$$\|T\|_{max}=\sup\left\{\|\phi (T)\|_{\mathnormal{B}(H_\phi)}\mid \phi:\IC_u[X,\mathcal{A}] \rightarrow \mathcal{B}(H_\phi) \text{ a } *\text{-representation}\right\}.$$
\end{Def}

By the universality of the maximal completion, we have the canonical quotient
$$\lambda: C_{u,max}^*(X,\mathcal{A}) \rightarrow C_u^*(X,\mathcal{A})$$
induced by the $*$-representation of $\mathbb{C}_u[X,\mathcal{A}]$ on the Hilbert module $E$.

Analogous to Definition \ref{algebraic twisted uniform Roe algebra}, we can define the {\em  algebraic twisted Roe algebra}, denoted by   $\IC[X,\mathcal{A}\hat{\otimes}\mathcal{K}]$, which consists of all functions 
$$T: Y \times Y\rightarrow \mathcal{A}(H)\hat{\otimes}\mathcal{K}$$ 
satisfying the conditions as in Definition \ref{algebraic twisted uniform Roe algebra}. 
Then we can similarly define the {\em maximal and reduced twisted Roe algebras}, denoted $C_{max}^*(X,\mathcal{A}\hat{\otimes}\mathcal{K})$ and $C^*(X,\mathcal{A}\hat{\otimes}\mathcal{K})$, respectively, as relevant completions of the  algebraic twisted Roe algebra.

In the rest part of this section, we aim to prove the following result.

\begin{Thm}\label{max-red}
{\it
The canonical quotient maps $$\lambda: C_{u,max}^*(X,\mathcal{A}) \rightarrow C_u^*(X,\mathcal{A})$$
and 
$$\lambda: C_{max}^*(X,\mathcal{A}\hat{\otimes}\mathcal{K}) \rightarrow C^*(X,\mathcal{A}\hat{\otimes}\mathcal{K})$$ are both isomorphisms.
}
\end{Thm}

To begin with, we analyze ideals of the twisted (uniform) Roe algebras associated with open subsets of $\IR_+\times H$.

\begin{Def}\leavevmode
(1) The support of an element $T$ in $\IC_u[X,\mathcal{A}]$ is defined to be
		$$\left\{(x,y,t,h)\in X\times X \times \mathbb{R}_+ \times H\mid(t,h)\in \supp(T(x,y))\right\}.$$
(2) Let $O$ be an open subset of  $\mathbb{R}_+ \times H$. Define $C_u[X,\mathcal{A}]_{O}$ to be the $*$-subalgebra of $\mathbb{C}_u[X,\mathcal{A}]$ consisting of all elements whose supports are contained in $X\times X \times \mathnormal{O}$, i.e.,
		$$\mathbb{C}_u[X,\mathcal{A}]_{O}=\left\{T\in \IC_u[X,\mathcal{A}]\mid\supp(T(x,y))\subseteq \mathnormal{O}, \forall x,y\in X\right\}.$$
\end{Def}

Then $\mathbb{C}_u[X,\mathcal{A}]_{O}$ is a $*$-subalgebra of the maximal twisted uniform Roe algebra $C^*_{u,max}(X,\A)$. Denoted by $C_{u,\phi}^*(X,\mathcal{A})_{O}$ the norm closure of $\mathbb{C}[X,\mathcal{A}]_{O}$ under the norm in $C_{u,max}^*(X,\mathcal{A})$ (which is not the maximal norm completion of $\mathbb{C}_u[X,\mathcal{A}]_{O}$). We can also view $\mathbb{C}_u[X,\mathcal{A}]_{O}$ as a $*$-subalgebra of the reduced twisted uniform Roe algebra $C^*_u(X,\A)$. Denoted by $C_u^*(X,\mathcal{A})_{O}$ the norm closure of $\mathbb{C}_u[X,\mathcal{A}]_{O}$ under the norm in $C_u^*(X,\mathcal{A})$. Then the canonical quotient map restricts to a $*$-homomorphism
$$\lambda:C_{u,\phi}^*(X,\mathcal{A})_{O}\to C_u^*(X,\mathcal{A})_{O}.$$

We have the following lemma about the ideals, which can be proved in a similar way to \cite[Lemma 6.3]{Yu2000}.

\begin{Lem}\label{idealdecom}
{\it 
Let $B(f(p(x)),r)=\left\{(t,h)\in \mathbb{R}_+ \times H: t^2+\|h-f(p(x))\|< r^2\right\}$ for each $r> 0$ and $x\in X$. Let $X_{i,j}$ and $X'_{i,j}$ be subsets of $X$ for all $1\leq i \leq i_0$ and $1\leq j \leq j_0$, $\mathnormal{O}_{r,j}=\mathop\cap_{i=1}^{i_0}(\mathop \cup_{x \in X_{i,j}}B(f(p(x)),r))$ and $\mathnormal{O}'_{r,j}=\mathop\cap_{j=1}^{j_0}(\mathop \cup_{x \in X'_{i,j}}B(f(p(x)),r))$, then for each $r_0>0$ we have
\begin{itemize}
    \item [(1)] $\lim\limits_{r<r_0,r\rightarrow r_0}C_u^*(X,\mathcal{A})_{O_{r,j}\cup O'_{r,j}}=\lim\limits_{r<r_0,r\rightarrow r_0}C_u^*(X,\mathcal{A})_{O_{r,j}}+\lim\limits_{r<r_0,r\rightarrow r_0}C_u^*(X,\mathcal{A})_{O'_{r,j}},$
\item [(2)] $\lim\limits_{r<r_0,r\rightarrow r_0}C_u^*(X,\mathcal{A})_{O_{r,j}\cap O'_{r,j}}=\lim\limits_{r<r_0,r\rightarrow r_0}C_u^*(X,\mathcal{A})_{O_{r,j}}\cap \lim\limits_{r<r_0,r\rightarrow r_0}C_u^*(X,\mathcal{A})_{O'_{r,j}},$
\item [(3)] $\lim\limits_{r<r_0,r\rightarrow r_0}C_{u,\phi}^*(X,\mathcal{A})_{O_{r,j}\cup O'_{r,j}}=\lim\limits_{r<r_0,r\rightarrow r_0}C_{u,\phi}^*(X,\mathcal{A})_{O_{r,j}}+\lim\limits_{r<r_0,r\rightarrow r_0}C_{u,\phi}^*(X,\mathcal{A})_{O'_{r,j}},$
\item [(4)] $\lim\limits_{r<r_0,r\rightarrow r_0}C_{u,\phi}^*(X,\mathcal{A})_{O_{r,j}\cap O'_{r,j}}=\lim\limits_{r<r_0,r\rightarrow r_0}C_{u,\phi}^*(X,\mathcal{A})_{O_{r,j}}\cap \lim\limits_{r<r_0,r\rightarrow r_0}C_{u,\phi}^*(X,\mathcal{A})_{O'_{r,j}}.$
\end{itemize}
}
\end{Lem}

\medskip
\par
\begin{Def}\label{r-sep}
Let $r>0$. A family of open subsets $\{O_{r,i}\}_{i\in J}$ is said to be {\em $r$-separate} if:
\begin{itemize}
 \item[(1)] for each $i\in J$, there exists $x_i\in X$ such that $\mathnormal{O}_{r,i} \subseteq B(f(p(x_i)),r)$, where
 $$B(f(p(x_i)),r)=\left\{(t,h)\in \mathbb{R}_+\times H: t^2+\|h-f(p(x_i))\|^2<r^2\right\},$$
 and $f$ is the coarse embedding $f:Y\rightarrow H$;
 \item[(2)] $\mathnormal{O}_{r,i} \cap \mathnormal{O}_{r,j} =\emptyset$ if $i\neq j.$
\end{itemize}

\end{Def}

\begin{Pro}\label{local iso}
{\it 
Let $r>0$. 
If $\{O_{r,i}\}_{i\in J}$ is a family of $r$-separate open subsets of $\mathbb{R}_+\times H$, then the canonical quotient map
$$\lambda: C_{u,\phi}^*(X,\mathcal{A})_{O_r} \rightarrow C_u^*(X,\mathcal{A})_{O_r} $$
is an isomorphism, where $O_r$ is the disjoint union of $\{O_{r,i}\}_{i\in J}$.
}
\end{Pro}

We will prove Proposition \ref{local iso} in the next half of this section. Granting Proposition \ref{local iso} for the moment, we are able to prove Theorem \ref{max-red}.

\begin{proof}[Proof of Theorem \ref{max-red}]
For any $r>0$, we define $O_r\subseteq \mathbb{R}_+\times H$ by
$$O_r=\bigcup_{y \in Y}B(f(y),r),$$
where $f:Y \rightarrow H$ is the coarse embedding and
$$B(f(y),r)=\{(t,h)\in \mathbb{R}_+\times H:t^2+\|h-f(y)\|^2<r\}.$$
By the definition of twisted uniform Roe algebra, we have that
$$C_u^*(X, \mathcal{A})=\lim\limits_{r\rightarrow \infty}C_u^*(X,\mathcal{A})_{\mathnormal{O_r}} \quad \text{ and }\quad C_{u,max}^*(X, \mathcal{A})=\lim\limits_{r\rightarrow \infty}C_{u,\phi}^*(X,\mathcal{A})_{\mathnormal{O_r}}.$$
So it suffices to show that for any $r>0$,
$$\lambda: C^*_{u,\phi}(G, \mathcal{A})_{\mathnormal{O_r}}\rightarrow C_u^*(X, \mathcal{A})_{\mathnormal{O_r}}$$
is an isomorphism.

Since $Y$ has bounded geometry and $f:Y\rightarrow H$ is a coarse embedding, then for given $r>0$, there exist finitely many, say $k_r$, mutually disjoint subsets $Y_k$ of $Y$ such that $Y=\bigsqcup_{k=1}^{k_r}Y_k$, where for each k, $d(f(y_i^k),f(y_j^k))>3r$ for distinct element $y_i^k,y_j^k$ in $Y_k$.
	
	Let $\mathnormal{O_{r,k}}=\bigsqcup_{y^k\in X_k}B(f(y^k),r)$. Then $\mathnormal{O_r}=\bigcup_{k=1}^{k_r}\mathnormal{O_{r,k}}$.
	By Lemma \ref{idealdecom}, we know that the following diagrams
	
	\begin{equation*}
	\xymatrix {
	\lim\limits_{r<r_0,r\rightarrow r_0}C_u^*(X,\mathcal{A})_{\cup_{k=1}^{n}\mathnormal{O_{r,k}}\cap \mathnormal{O_{r,n+1}}}\ar[r]\ar[d] & \lim\limits_{r<r_0,r\rightarrow r_0}C_u^*(X,\mathcal{A})_{\cup_{k=1}^{n}\mathnormal{O_{r,k}}}\ar[d]\\
	  \lim\limits_{r<r_0,r\rightarrow r_0}C_u^*(X,\mathcal{A})_{\mathnormal{O_{r,n+1}}}\ar[r] & \lim\limits_{r<r_0,r\rightarrow r_0}C_u^*(X,\mathcal{A})_{\cup_{k=1}^{n+1}\mathnormal{O_{r,k}}}
}
	\end{equation*}
and
\begin{equation*}
\xymatrix {
	\lim\limits_{r<r_0,r\rightarrow r_0}C_{u,\phi}^*(X,\mathcal{A})_{\cup_{k=1}^{n}\mathnormal{O_{r,k}}\cap \mathnormal{O_{r,n+1}}}\ar[r]\ar[d] & \lim\limits_{r<r_0,r\rightarrow r_0}C_{u,\phi}^*(X,\mathcal{A})_{\cup_{k=1}^{n}\mathnormal{O_{r,k}}}\ar[d]\\
  \lim\limits_{r<r_0,r\rightarrow r_0}C_{u,\phi}^*(X,\mathcal{A})_{\mathnormal{O_{r,n+1}}}\ar[r] & \lim\limits_{r<r_0,r\rightarrow r_0}C_{u,\phi}^*(X,\mathcal{A})_{\cup_{k=1}^{n+1}\mathnormal{O_{r,k}}}
}
\end{equation*}
are both pushout diagrams. Since every $O_{r,k}$ and $O_{r,n}\cap (\cup_{k=1}^{n}\mathnormal{O_{r,k}})$ are $r$-separate for $1\leq n\leq k_r$, and the $C^*$-norm on a pushout is uniquely determined by the norms on the other three $C^*$-algebras in the diagram
	defining it, by Proposition \ref{local iso}, it follows that
	the canonical quotient map $\lambda: C_{u,max}^*(X,\mathcal{A}) \rightarrow C_u^*(X,\mathcal{A})$ is an isomorphism.  The isomorphism from $C_{max}^*(X,\mathcal{A}\hat\otimes\mathcal{K})$ to $ C^*(X,\mathcal{A}\hat\otimes\mathcal{K})$ can be proved similarly.
\end{proof}

We need some preparations before we can prove Proposition \ref{local iso}. We shall first introduce two $C^*$-algebras for a collection of metric spaces.

Let $\{O_{r,i}\}_{i\in I}$ and $\{x_i\}_{i\in I}$ be as in Definition \ref{r-sep}, and let $\{X_i\}_{i\in I}$ be a collection of subspaces of $X$ with equi-property $A$ such that $x_i\in X_i$ for every $i\in I$. Define a $*$-algebra $\IA[X_i,i\in I]$ to be the subalgebra of $\prod\limits_{i\in I}\mathbb{C}_u[X_i]\wod \mathcal{A}_{O_{r,i}}$ consisting of all elements $\oplus_{i\in I}T_i$ satisfying the follow conditions:
\begin{enumerate}
\item[(1)]there exists $S>0$, such that $\mbox{propagation}(T_i)<S$ for all $i\in I$;
\item[(2)]there exists $M>0$, such that $\sup_{i}\sup_{x,y\in X_i}\|T_i(x,y)\|<M$;
\item[(3)]there is $T_{1,i}\in \mathbb{C}_u[X_i]\wod\mathcal{A}(W_N(x_i))$ for which $(\beta_N(x_i))(T_{1,i})=T_i$;
\item[(4)]$D_Z(T_{1,i})$ exists in $\mathbb{C}_u[X_i]\wod\mathcal{A}(W_N(x_i))$ and there exists $c>0$ such that $\|D_Z(T_{1,i})\|\leq c$ for all $Z=(t,h)\in \mathbb{R}_+\times W_N(x_i)$ satisfying $\|Z\|\leq 1$, where $\beta_N(x_i)$ is the $*$-homomorphism:
$$\mathbb{C}_u[X_i]\wod\mathcal{A}(W_N(x_i)) \rightarrow \mathbb{C}_u[X_i]\wod\mathcal{A}(V)$$
induced by the inclusion of $W_N(x_i)$ into $V$.
\end{enumerate}

Notice that $\IA[X_i,i\in I]$ has a natural $*$-representation on $\oplus_{i\in I}\ell^2(X_i)\wox\A_{O_{r,i}}$ by matrix multiplication. Define $A_u^*(X_i,i\in I)$ to be the completion of $\IA[X_i,i\in I]$ under the norm induced by the $*$-representation as above. Since the sequence $\{X_i\}_{i\in I}$ has bounded geometry, it is reasonable to define the maximal completion of $\IA[X_i,i\in I]$, denoted by $A^*_{u,max}(X_i,i\in I)$, as the completion under the norm
$$\|(T)\|=\sup\left\{\|\pi(T)\|_{B(H)}\,\Big|\, \pi:\IA[X_i,i\in I]\rightarrow B(H) \text{ a $*$-representation }\right\}.$$

Since $\{X_i\}_{i\in I}$ has equi-property A, for any $R>0$, $\e>0$, there exist $S>0$ and a family of maps $\xi^{(i)}: X_i \rightarrow \ell^2(X_i)$ such that
(1) for all $x,y\in X_i$, $\xi^{(i)}_x(y)\in [0,1]$;
(2) for all $x\in X_i$, $\|\xi^{(i)}_x\|=1$;
(3) $\sup\{\|\xi^{(i)}_x-\xi^{(i)}_y\|\mid d(x,y)\leq R, x,y \in X_i\}<\e$;
(4) for all $x\in X_i$, $\mbox{supp} (\xi^{(i)}_x)\subset B(x,S)$. 

The family of maps $\{\xi^{(i)}\}_{i\in I}$ induce:
\begin{itemize}
\item[(i)] a family of partitions of unity $\{\phi^{(i)}_y\}_{y\in X_i}$ on $X_i$ defined by $\phi^{(i)}_y(x)=\xi^{(i)}_x(y)$;
\item[(ii)] a family of kernels $k^{(i)}: X_i \times X_i \rightarrow [0,1]$ defined by $k^{(i)}(x,y)=\langle\xi^{(i)}_x,\xi^{(i)}_y\rangle$;
\item[(iii)] a family of Schur multipliers $M^{alg}_{k^{(i)}}: \mathbb{C}_u[X_i] \rightarrow \mathbb{C}_u[X_i]$ defined by $$(M^{alg}_{k^{(i)}}T)(x,y)=k^{(i)}(x,y)T(x,y).$$	
\end{itemize}

For $(T_i)_{i\in I}=\left(\sum_{m=1}^{n}T_{i,m}\wox a_{i,m}\right)_{i\in I} \in \IA[X_i,i\in I]$, we define
\begin{equation*}
M_k^{alg}((T_i)_i)=\left(\sum_{m=1}^{n}M^{alg}_{k^{(i)}}(T_{i,m})\wox a_{i,m}\right)_{i\in I}=\left(\sum_{m=1}^{n}\sum_{x\in X_i}\phi_x^{(i)}T_{i,m}\phi^{(i)}_x\wox a_{i,m}\right)_{i\in I}.
\end{equation*}
The sum converges in the strong operator topology. 

Given a sequence of points $x=\{x_i\}_{i\in I}$ where $x_i\in X_i$, denote by $\phi_x=(\phi_{x_i}^{(i)}\wox 1)_{i\in I}$ the collection of the functions $\phi_{x_i}^{(i)}:X_i \to [0,1]$  on the first component and identity on the second component. Then $\phi_x$ can be viewed as a bounded operator defined by multiplication on $\oplus_{i\in I}\ell^2(X_i)\wox\A_{O_{r,i}}$. We say two sequences of points $x=\{x_i\}_{i\in I}$ and $y=\{y_i\}_{i\in I}$ are distinct if $x_i \neq y_i$ for any $i\in I$. Then for $T=(T_i)_{i\in I}\in \IA[X_i,i\in I]$, $$M^{alg}_k(T)=\sum_{x}\phi_xT\phi_x,$$ where $x$ ranges over all distinct sequences of points in $\{X_i\}_{i\in I}$.

Naturally, the map $M^{alg}_k(T)$ can extend to  a unital completely positive map from $A_u^*(X_i,i\in I)$ to itself. We would like to extend $M_k^{alg}$ to $A^*_{u,max}(X_i,i\in I)$, unfortunately, the sum in the above formula makes no sense here. Instead, we can use the following several general results by similar arguments in \cite[Section 2]{JR2013}.

\begin{Lem}\label{finite sum}
For any $S_0>0$, there exists a collection of functions $\phi_j=(\phi^{(i)}_j\wox1)_{i\in I}$, where $\phi_j^{(i)}:X_i\rightarrow [0,1], j=1,\cdots, N$ such that for all $T=(T_i)_{i\in I}\in \IA[X_i,i\in I]$ with $\textup{propagation}(T)\leq S_0$,
$$M_k^{alg}(T)=\sum_{j=1}^N\phi_jT\phi_j,$$
where $T$ is seen as an operator on $\oplus_{i\in I}\ell^2(X_i)\wox\A_{O_{r,i}}$.
\end{Lem}

\begin{proof}
Let $S>0$ be the support bound for $\{\xi^{(i)}\}_{i\in I}$. Since the sequence $\{X_i\}_{i\in I}$ has uniformly bounded geometry, it is well-known that there exists an integer $N$ such that for each $i\in I$, $X_i$ can be decomposed into $X_i=\bigsqcup_{n=1}^N X^{(i)}_n$ such that for any $x,y\in X^{(i)}_n$ with $x\neq y$, $d(x,y)>2S+S_0$. For each $j=1,\cdots, N $, define $\phi^{(i)}_j=\sum_{x}\phi^{(i)}_x,$ where $\{\phi^{(i)}_x\}$ is the partition of unity associated to $\xi^{(i)}$, and $x$ ranges over all distinct sequences of points in $\{X_j^{(i)}\}_{i\in I}$. one can check that $\phi_j=\{(\phi^{(i)}_j\wox1)_{i\in I}\}_{j=1}^N$ has the property claimed.
\end{proof}

\begin{Cor}\label{extend}
The associated linear map $M_{k}^{alg}:\IA[X_i,i\in I]\to\IA[X_i,i\in I]$ extends to a unital completely positive(u.c.p) map on any $C^*$-algebraic completion of $\IA[X_i,i\in I]$.
\end{Cor}

\begin{proof}
Given $T\in M_n(\IA[X_i,i\in I])$. Let $S_0=\mbox{propagation}(T)$. According to Lemma \ref{finite sum}, there exist $\phi_j$, $j=1,\cdots, N$ such that
	$$M_k^{n}(T)=\sum_{j=1}^N \mbox{diag}(\phi_j)T \mbox{diag}(\phi_j)$$
where $M_k^{n}$ is the matrix version of $M_k^{alg}$ and $\mbox{diag}(\phi_j)$ is the diagonal matrix with all non-zero entries $\phi_j$. Thus $M_k^{alg}$ is a bounded u.c.p map. Then it can extend to the u.c.p map from any $C^*$-algebraic completion of $\IA[X_i,i\in I]$.
\end{proof}

We now consider the diagram
$$\xymatrix{\IA[X_i,i\in I]\ar[r]^{M^{alg}_k}\ar[d]&\IA[X_i,i\in I]\ar[d]\\
A^*_{u,max}(X_i,i\in I)\ar[r]^{M^{max}_k}\ar[d]^{\lambda}&A^*_{u,max}(X_i,i\in I)\ar[d]\\
A_u^*(X_i,i\in I)\ar[r]^{M_k}&A_u^*(X_i,i\in I)}$$
where the maps $M^{max}_k$ and $M_k$ are the u.c.p. maps given by the Corollary \ref{extend}.

\begin{Lem}\label{closed}
For any $S>0$, $\mathbb{A}_u^{S}[X_i,i\in I]$ which denoted the set of all $T\in\IA[X_i,i\in I]$ with $\textup{propagation}(T)\leq S$, is closed in $A^*_{u,max}(X_i,i\in I)$.

In particular, the image of Schur multiplier $M_k^{alg}$ associated to $\xi^{(i)}$ with support bound $S$ is contained in $\mathbb{A}_u^{2S}[X_i,i\in I]$.
\end{Lem}

\begin{proof}
Let $\{T_n\}_n$ be a sequence in $\mathbb{A}_u^{S}[X_i,i\in I]$ converging in $A^*_{u,max}(X_i,i\in I)$, then $\lambda(T_n) \rightarrow \lambda(T)$ in $A_u^*(X_i,i\in I)$, moreover the matrix entries of $\{T_n\}_n$ converge uniformly to those of $\lambda(T)$. According to Lemma \ref{maxnorm}, it follows that $\{T_n\}_n$ converges to $\lambda(T)$ in the $A^*_{u,max}(X_i,i\in I)$.

The remaining comment follows as $M^{max}_k(\IA[X_i,i\in I])\subseteq \mathbb{A}_u^{2S}[X_i,i\in I]$ under the stated assumptions.
\end{proof}

\begin{Pro}\label{equi-iso}
With assumption as above, the canonical quotient map
$$\lambda:A^*_{u,max}(X_i,i\in I)\to A_u^*(X_i,i\in I)$$
is an isomorphism.
\end{Pro}

\begin{proof}
As the sequnence $\{X_i\}_{i\in I}$ has equi-property A and Lemma \ref{maxnorm}, there exists a sequence of kernels $(k_n^{(i)})$ such that the associated Schur multiplier $M_{k_n}^{\lambda}$ and $M_{k_n}^{max}$ converge point-norm to the identity on $A_u^*(X_i,i\in I)$ and $A^*_{u,max}(X_i,i\in I)$, respectively. If $T\in A^*_{u,max}(X_i,i\in I)$ is in the kernel of $\lambda$, then we have that $\lambda(M_{k_n}^{max}(T))=M_{k_n}^{\lambda}(\lambda(T))=0$ for all $n$.

However, Lemma \ref{closed} implies that $M_{k_n}^{max}(T)\in \IA[X_i,i\in I]$ for all $n$. As $\lambda$ is injective here, it follows that $M_{k_n}^{max}(T)=0$ for all $n$, therefore $T=0$. Hence $\lambda$ is injective as desired.
\end{proof}

\begin{proof}[Proof of Proposition \ref{local iso}]
As $O_r$ is the disjoint union of a family of open subsets $\{O_{r,i}\}_{i\in I}$ in $\mathbb{R}_+\times H$, for any $T \in \mathbb{C}_u[X,\mathcal{A}]_{O_r}$, Lemma \ref{idealdecom} implies there is a unique decomposition $T=\sum_i T_i$, where $\mbox{supp}(T_i)\subseteq X\times X \times\mathnormal{O}_{r,i}$. Then there is a $*$-homomorphism
\begin{equation*}
\begin{aligned}
\Phi: \IC_u[X,\A]_{O_r}&\to\prod_{i\in I}\IC_u[X,\A]_{O_{r,i}} \\
T&\mapsto (T_i).
\end{aligned}
\end{equation*}

Since $f: Y\to H$ is a coarse embedding, by the condition (4) of Definition \ref{algebraic twisted uniform Roe algebra}, it follows that for each $T \in \mathbb{C}_u[X,\mathcal{A}]_{O_r}$,  there exists $m\in\mathbb{N}$ depending on $T$ such that $$T\mapsto(T_i)_{i\in I}\in\IA\left[p^{-1}\left(B(p(x_i),m)\right),i\in I\right].$$ Then we actually have that
$$\IC_u[X,\A]_{O_r}=\bigcup_{m=1}^{\infty}\IA\left[p^{-1}\left(B(p(x_i),m)\right),i\in I\right],$$
which gives rise to an identification in the reduced case
$$C^*_u(X,\A)_{O_r}=\lim_{m\to\infty}A_u^*\left(p^{-1}\left(B(p(x_i),m)\right),i\in I\right),$$
and 
$$C^*_{u,\phi}(X,\A)_{O_r}=\lim_{m\to\infty}A_{\phi}^*\left(p^{-1}\left(B(p(x_i),m)\right),i\in I\right),$$
where $A_{\phi}^*\left(p^{-1}\left(B(p(x_i),m)\right),i\in I\right)$ is the completion of $\IA\left[p^{-1}\left(B(p(x_i),m)\right),i\in I\right]$ under the norm in $C^*_{u,\phi}(X,\A)_{O_{r}}$.

By the condition (2) of the map $p$, the collection $p^{-1}\left(B(p(x_i),m)\right)$ has equi-property A for each $m>0$. According to Proposition \ref{equi-iso}, it follows that
$$\lambda:A^*_{u,max}(p^{-1}\left(B(p(x_i),m)\right),i\in I)\to A_u^*(p^{-1}\left(B(p(x_i),m)\right),i\in I)$$
is an isomorphism.

Note that the norm on $A^*_{\phi}(p^{-1}\left(B(p(x_i),m)\right),i\in I)$ stands between the maximal norm and the reduced norm, so for each fixed $m$, there is an $*$-isomorphism
$$A^*_{\phi}(p^{-1}\left(B(p(x_i),m)\right),i\in I)\cong A_u^*(p^{-1}\left(B(p(x_i),m)\right),i\in I).$$
Consequently, the canonical quotient map
$$\lambda: C_{u,\phi}^*(X,\mathcal{A})_{\mathnormal{O_r}} \rightarrow C_u^*(X,\mathcal{A})_{\mathnormal{O_r}} $$
is an isomorphism.
\end{proof}

\section{The proof of the main theorem}

In \cite{Yu2000}, G. Yu proved a geometric analogue of the infinite-dimensional Bott periodicity and used it to
reduce the coarse Baum--Connes conjecture to the twisted coarse Baum--Connes conjecture. In this section, we will use the geometric analogue of the infinite-dimensional Bott periodicity to prove Theorem \ref{main result}.

Let $H$ be a separable and infinite-dimensional Hilbert space. Let $V$ be the countably infinite-dimensional Euclidean subspace of $H$ defined in Section \ref{Section 3}. If $V_a\subset V$ is a finite-dimensional affine subspace, then we use $\mathcal{H}_a=L^2(V_a,\Cl(V_a^0))$ to denote the Hilbert space of square integrable functions from $V_a$ to $\text{Cliff}(V_a^0)$.
If $V_a\subset V_b$, then there exists an isomorphism $$\mathcal{H}_b\cong \mathcal{H}_{ba}\wox\mathcal{H}_a,$$
where $\mathcal{H}_{ba}$ is the the Hilbert space associated with $V_{ba}^0$.
Define a unit vector $\xi_{ba}\in \mathcal{H}_{ba}$ by
$$\xi_{ba}(v_{ba})=\pi^{-\text{dim}(V_{ba}^0)/{4}}\exp\left(-\frac{1}{2}\|v_{ba}\|^2\right)$$
for all $v_{ba}\in V_{ba}^0$. Regarding $\mathcal{H}_{a}$ as a subspace of $\mathcal{H}_{b}$ via the isometry $\xi \mapsto \xi_0\hat{\otimes} \xi$, we define
$$\H=\lim\limits_{\longrightarrow}\mathcal{H}_{a}.$$
where the Hilbert space direct limit is taken over the direct system of finite-dimensional affine subspaces of $V$.

Let $s=\lim\limits_{\longrightarrow}s_a$ be the algebraic direct limit of the Schwartz subspaces $s_a\subset \H_{a}$. If $V_a\subset V$ is a finite-dimensional affine subspace, then the Dirac operator $D_a$ is an unbounded operator on $\H_a$ with domain $s_a$, defined by 
$$D_a \xi=\sum \limits_{i=1}^n (-1)^{\text{deg}\xi} \frac{\partial\xi}{\partial x_i} v_i$$
where $\{v_1,\cdots,v_n\}$ is an orthonormal basis for $V_a^0$ and $\{x_1,\cdots,x_n\}$ are the dual coordinates to $\{v_1,\cdots, v_n\}$.
If $V_a\subset V$ is a linear subspace, we define the Clifford operator $C_a$ on $s$ by
$$(C_a\xi)(v_b)=v_a\cdot\xi(v_b)$$
for any $\xi \in s_b$ and $v_b \in V_b$, where $v_a$ is the vector projection of $v_b$ onto $V_a$, and $v_a\cdot \xi(v_b)$ is the Clifford multiplication of $v_a$ and $\xi(v_b)$. Note that $C_a$ is a well-defined operator on $s$.

For any $x\in X$, let $V_n(p(x))=W_{n+1}(p(x))\ominus W_n(p(x))$ if $n\geq 1$, $V_0(p(x))=W_1(p(x))$, where $x\in X$ and $W_n(p(x))$ is as in Section \ref{Section 3}. We have the algebraic decomposition:
$$V=V_0(p(x))\oplus V_1(p(x))\oplus \cdots \oplus V_n(p(x))\oplus \cdots$$
For each $n$, we define an unbounded operator $B_{n,t}(x)$ on $\mathcal{H}$ as follows:
$$B_{n,t}(x)=t_0D_0+t_1D_1+\cdots+t_{n-1}D_{n-1}+t_n(D_n+C_n)+t_{n+1}(D_{n+1}+C_{n+1})+\cdots$$
where $t_j=1+t^{-1}j$, $D_n$ and $C_n$ are respectively the Dirac operator and Clifford operator associated to $V_n(p(x))$. By Lemma 5.8 in \cite{HKT1998}, $B_{n,t}(p(x))$ is essentially self-adjoint with compact resolvent.

Now, we are ready to construct an asymptotic morphism $(\alpha_t)_{t\in [1,\infty)}$ from $C^*_{u}(X, \mathcal{A})$ to $\mathcal{S}\wox UC^*(X)$. For every non-negative integer $n$ and $x\in X$, we defined
$$\theta_t^n(p(x)): \mathcal{A}(W_n(p(x)))\rightarrow \mathcal{S}\wox K(\mathcal{H})$$ by
$$\theta_t^n(p(x))(g\wox h)=g_t(X\wox 1+1\wox B_{n,t}(x))(1\wox M_{h_t}),$$
for each $g\in \mathcal{S}$, $h\in \mathcal{C}(W_n(p(x)))$, where $g_t(s)=g(t^{-1}s)$ for all $t\geq 1$ and $s\in \mathbb{R}$, $h_t(v)=h(t^{-1}v)$ for all $t\geq 1$ and $v\in W_n(p(x))$, $M_{h_t}$ acts on $\mathcal{H}$ by pointwise multiplication. Lemma 5.8 in \cite{HKT1998} implies that $\theta_t^n(x)(g\hat{\otimes}h) \in \mathcal{S}\hat{\otimes}K(\mathcal{H})$.

\begin{Rem}\label {asymptotic momorphism}
By the proof of Proposition 4.2 in \cite{HKT1998}, we have the following asymptotically commutative diagram:
$$\xymatrix{\mathcal{A}(W_n(p(x))) \ar[d]_{\beta_{n+1,n}(p(x))}\ar[drr]^{\theta_t^n(x)}\\
\mathcal{A}(W_{n+1}(p(x)))  \ar[rr]^{\theta_t^{n+1}(x)} & & \mathcal{S} \hat{\otimes}K(H) }$$
where $\beta_{n+1,n}(p(x))$ is the $*$-homomorphism associated to the inclusion $W_n(p(x)\to W_{n+1}(p(x)))$.
\end{Rem}

\begin{Def}
For each $t\in[1,\infty)$, we define a map
$$\alpha_t:\mathbb{C}_u[X, \mathcal{A}] \rightarrow \mathcal{S}\wod U\mathbb{C}[X]$$
by:
$$(\alpha_t(T))(x,y)=(\theta_t^N(p(x)))(T_1(x,y))$$
for every  $T \in \mathbb{C}_u[X, \mathcal{A}]$, where $N$ is a non-negative integer such that for every pair $(x,y) \in X\times X$, there exists $T_1(x,y)\in \mathcal{A}(W_N(p(x))) $ satisfying $(\beta_N(x))(T_1(x,y))=T(x,y).$

As for the algebraic twisted Roe algebra, we can define a map 
$$\alpha_t:\mathbb{C}[X, \mathcal{A}\hat{\otimes}\mathcal{K}] \rightarrow \mathcal{S}\wod \mathbb{C}[X, \mathcal{K}]$$
by:
$$(\alpha_t(S))(x,y)=(\theta_t^N(p(x)))(S_1(x,y))$$
for every  $S \in \mathbb{C}[X, \mathcal{A}\hat{\otimes}\mathcal{K}]$, $t\in[1,\infty)$, where $N$ is a non-negative integer such that for every pair $(x,y) \in X\times X$, there exists $S_1(x,y)\in \mathcal{A}(W_N(p(x)))\hat{\otimes}\mathcal{K} $ satisfying $(\beta_N(x))(S_1(x,y))=S(x,y).$
\end{Def}

With the same argument in \cite[Lemma 4.4]{JR2013}, we see that the ranges of the maps $(\alpha_t)_{t\in [1,\infty)}$ are contained in $U\mathbb{C}[X]$, thus the maps is well-defined.  By Remark \ref{asymptotic momorphism}, if $N'>N$,  we know that
$$\|(\theta_t^{N'}(x)) (\beta_{N',N}(p(x))(T_1(x,y)))-\theta_t^N(x)(T_1(x,y))\|$$
uniformly tends to $0$ for all $x\in X$ as $t \rightarrow \infty$. Thus we have that $\alpha_t(T)$ does not asymptotically depend on the choice of $N$.

By Lemma 7.3, Lemma 7.4 and Lemma 7.5 in \cite{Yu2000}, we have the following result.

\begin{Lem}\label{Dirac}
{\it 
The maps $(\alpha_t)_{t\in [1,\infty)}$ extend to  asymptotic morphisms
$$(\alpha_t)_{t\in [1,\infty)}: C_u^*(X, \mathcal{A}) \rightarrow \mathcal{S}\wox UC^*(X)$$
and its maximal counterpart
$$(\alpha_t)_{t\in [1,\infty)}: C_{u,max}^*(X, \mathcal{A}) \rightarrow \mathcal{S}\wox UC_{max}^*(X).$$
}
\end{Lem}

\begin{proof}
Given any $T,S\in\mathbb{C}_u[X,\mathcal{A}]$. By the condition (1) of Definition \ref{A-by-CE coarse fibration} and the definition of $W_n(p(x))$, there exists $N>0$ such that $W_n(p(x))\subset W_{n+1}(p(y))$ for all $n\geq N$ and $x,y\in X$ satisfying $d_X(x,y)\leq \text{propagation}(T)$.  By the condition $(1)$ of Definition \ref{algebraic twisted uniform Roe algebra} and Remark \ref{composition}, there exists $N_0$ such that for each $n\geq N_0$, and for every pair $(x,y)\in X\times Y$ there exists $T_1(x,y)\in \mathcal{A}(W_{n}(p(x))) $ satisfying $(\beta_{n}(x))(T_1(x,y))=T(x,y)$ and $(\beta_{n}(x))(S_1(x,y))=S(x,y)$. It follows from Lemma 7.3, Lemma 7.4 and Lemma 7.5 in \cite{Yu2000} that
$$\big\|(\theta^n_t(p(x))(T_1S_1(x,y))-\theta^n_t(p(x))(T_1(x,y))\theta^n_t(p(x))(S_1(x,y))\big\|_{\mathcal{S}\wox\K(\H)}\to0$$
and
$$\|(\alpha_t(T^*)-\alpha_t(T)^*)(x,y)\|_{\mathcal{S}\wox\K(\H)}\to 0$$
uniformly for all $x,y\in X$ as $t\to\infty$. By an analogue of lemma \ref{maxnorm}, we have that
$$\|\alpha_t(ST)-\alpha_t(S)\alpha_t(T)\|_{max}\to0\quad\mbox{and}\quad\|(\alpha_t(T^*)-\alpha_t(T)^*)\|_{max}\to0$$
as $t\to\infty$. Then, by the universal property of the maximal norm, we have a well-defined asymptotic morphism
$$\alpha_t:C_{u,max}^*(X, \mathcal{A}) \rightarrow \mathcal{S}\wox UC_{max}^*(X).$$
By the composition with the canonical quotient map, we have that
$$\alpha_t:C_{u,max}^*(X, \mathcal{A})\to \mathcal{S}\wox UC_{max}^*(X)\to\mathcal{S}\wox UC^*(X).$$
Combining with Theorem \ref{max-red}, we have a well-defined asymptotic morphism
$$\alpha_t:C_u^*(X, \mathcal{A}) \rightarrow \mathcal{S}\wox UC^*(X).$$
This gives the desired result.
\end{proof}

As for the Roe algebras, we likewise have the following lemma.

\begin{Lem}
{\it
	The maps $(\alpha_t)_{t\in [1,\infty)}$ extend to asymptotic morphisms
	$$(\alpha_t)_{t\in [1,\infty)}: C^*(X, \mathcal{A}\hat{\otimes}\mathcal{K}) \rightarrow \mathcal{S}\wox C^*(X)$$
	and its maximal counterpart
	$$(\alpha_t)_{t\in [1,\infty)}: C_{max}^*(X, \mathcal{A}\hat{\otimes}\mathcal{K}) \rightarrow \mathcal{S}\wox C_{max}^*(X).$$
}
\end{Lem}

Next we shall construct an asymptotic morphism $(\beta_t)_{t\in [1,\infty)}$ from $\mathcal{S}\hat{\otimes}C_u^*(X)$ to $C_u^*(X,\mathcal{A})$.

For $x\in X$, let $E$ be a finite-dimensional Euclidean subspace of $V$ containing $f(p(x))$, where $f:X\rightarrow V$ is a coarse embedding. Define a $*$-homomorphism
$$\beta_{E}(x):\mathcal{S} \rightarrow \mathcal{A}(E)$$
as follows:
$$\beta_{E}(x)(g)=g(X\wox 1+1\wox C_{E,p(x)})$$
for all $g\in \mathcal{S}$, where $X$ is the degree one and unbounded multiplier of $\mathcal{S}$, $C_{E,p(x)}$ is the Clifford algebra-valued function on $E$ denoted by  
$$C_{E,p(x)}(v)=v-f(p(x))\in E \subset \text{Cliff}(E)$$ 
for all $v\in E$.

\begin{Def}
For each $t\in [1,\infty)$, we define a map
$$\beta_t:\mathcal{S}\hat{\otimes}\mathbb{C}_u[X] \rightarrow C_u[X,\mathcal{A}]$$
by
$$(\beta_t(g\hat{\otimes}T))(x,y)=\beta(x)(g_t)\hat{\otimes}T(x,y)$$
for all $g\in \mathcal{S}, T\in \mathbb{C}_u[X]$, where $g_t(s)=g(t^{-1}s)$, $s\in \mathbb{R}$, and $\beta(x): \mathcal{S} \rightarrow \mathcal{A}(V)$, is the $*$-homomorphism
associated to the inclusion of the zero-dimensional affine space $0$ into $V$ by
mapping $0$ to $f(p(x))$, where $\mathcal{A}(V)$ is defined in Section \ref{Section 3}.
\end{Def}
\par
As for the Roe algebras, analogous to the above definition, for each $t\in [1,\infty)$, we can define a map 
$$\beta_t:\mathcal{S}\hat{\otimes}\mathbb{C}[X, \mathcal{K}] \rightarrow \mathbb{C}[X,\mathcal{A}\hat{\otimes}\mathcal{K}].$$

Then we have the following lemma as in \cite{Yu2000}.
\begin{Lem}\label{Bott}
{\it
The maps $(\beta_t)_{t\in [1,\infty)}$ extend to asymptotic morphisms $$(\beta_t)_{t\in [1,\infty)}:\mathcal{S}\hat{\otimes}C_u^*(X)\rightarrow C_u^*(X,\mathcal{A}),$$
and its maximal counterpart
$$(\beta_t)_{t\in [1,\infty)}:\mathcal{S}\hat{\otimes}C_{u,max}^*(X)\rightarrow C_{u,max}^*(X,\mathcal{A}).$$
}
\end{Lem}

\begin{proof}
Let $E$ be the Hilbert module as defined in Section \ref{Section 3}. For given $g\in \mathcal{S}, T\in \mathbb{C}_u[X]$, note that $\beta_t(g\hat{\otimes} T)=N_{g_t}\cdotp (1\hat{\otimes}T)$ where $N_{g_t}:E\rightarrow E $ is denoted by
$$N_{g_t}\left(\sum_{x\in X_d}a_x[x]\right)=\sum_{x\in X_d}(\beta(x)(g_t)\hat{\otimes}1)a_x[x]$$  and $1\hat{\otimes} T: E\rightarrow E$ is defined by
$$(1\hat{\otimes} T)\left(\sum_{x\in X_d}a_x[x]\right)=\sum_{y\in X_d}\left(\sum_{x\in X_d}(1\hat{\otimes}T(y,x))a_x[x]\right)[y]$$
for all  $\sum_{x\in X_d}a_x[x]\in E$.
Since $\|N_{g_t}\|\leq \mathop{sup}\limits_{x\in X_d}\|\beta(x)(g_t)\|\leq \|g\|$, we have that $$\|\beta_t(g\hat{\otimes}T)\|\leq \|g\|\|T\|.$$
Thus, a family of maps $(\beta_t)_{t\in [1,\infty)}$ extends to a linear map from algebraic tensor product $\mathcal{S}\wod\mathbb{C}_u[X]$ to both $\mathcal{S}\wox_{max} C_u^*(X)$ and $\mathcal{S}\wox_{max} C^*_{u,max}(X)$.

To prove that $(\beta_t)_{t\in [1,\infty)}$ defines an asymptotic morphism, we need to prove
$$\|\beta(x)(g_t)-\beta(y)(g_t)\|$$
uniformly tends to $0$  on $\{(x,y)\in X\times X : d_X(x,y)\leq r\}$ for fixed $g\in C_0(\mathbb{R})$ and $r\geq 0$. Since $f$ is a coarse embedding, by the condition (1) of $p$, one can check $g$ with the generators $g(s)=\frac{1}{s\pm i}$.

Consequently, the maps $(\beta_t)_{t\in [1,\infty)}$ extend to an asymptotic morphism from the maximal tensor product $\mathcal{S}\hat{\otimes}_{max}C_u^*(X)$ to $C_u^*(X,\mathcal{A})$ . Since $\mathcal{S}$ is nuclear, we conclude that $(\beta_t)_{t\in [1,\infty)}$ extend to an asymptotic morphism from  $\mathcal{S}\hat{\otimes}C_u^*(X)$ to $C_u^*(X,\mathcal{A})$.

For the maximal case, $(\beta_t)_{t\in [1,\infty)}$ extends to an asymptotic morphism from $\mathcal{S}\hat{\otimes} C^*_{u,max}(X)$ to $C_{u,max}^*(X,\mathcal{A})$ for the same reason, as desired.
\end{proof}

As for the Roe algebras, like the argument above, we also have the following result.
\begin{Lem}
{\it 
	The maps $(\beta_t)_{t\in [1,\infty)}$ extend to asymptotic morphisms 
	$$(\beta_t)_{t\in [1,\infty)}:\mathcal{S}\hat{\otimes}C^*(X)\rightarrow C^*(X,\mathcal{A}\hat{\otimes}\mathcal{K}),$$
	and its maximal counterpart
	$$(\beta_t)_{t\in [1,\infty)}:\mathcal{S}\hat{\otimes}C_{max}^*(X)\rightarrow C_{max}^*(X,\mathcal{A}\hat{\otimes}\mathcal{K}).$$
}
\end{Lem}

Recall that in Remark \ref{Morita equivalence}, we state that 
J. \v{S}pakula and R. Willett \textup{\cite[Proposition 4.7]{JR2013}} constructed a pre-Hilbert $U\mathbb{C}[X]$-$\mathbb{C}_u[X]$-bimodule $E_{alg}$,  and then established the Morita equivalence between the (maximal) uniform algebra and the (maximal) uniform Roe algebra via different completion of pre-Hilbert module $E_{alg}$. By the result of R.Exel \cite{Exel1993}, a Morita equivalence bimodule $E$ induces an isomorphism $E_*$ on $K$-theory. The following result is a geometric analogue of the infinite-dimensional Bott periodicity introduced by N. Higson, G. Kaspatov and J. Trout in \cite{HKT1998}.  It is essentially proved in \cite[Proposition 7.7]{Yu2000}.

\begin{Pro}\label{composition iso}
{\it 
The compositions
$$E_*\circ\alpha_* \circ \beta_*: K_*(\mathcal{S}\hat{\otimes}C_{u}^*(X)) \rightarrow K_*(\mathcal{S}\hat{\otimes}C_{u}^*(X))$$
and its maximal counterpart
$$E_*\circ\alpha_* \circ \beta_*: K_*(\mathcal{S}\hat{\otimes}C_{u,\max}^*(X)) \rightarrow K_*(\mathcal{S}\hat{\otimes}C_{u,\max}^*(X))$$
are identity maps.
}
\end{Pro}

\begin{proof}
We will only prove it for the reduced case. The argument is the same for the maximal case. Let $g\in \mathcal{S}=C_0(\mathbb{R})$ be a continuously differentiable function with compact support, and $T\in \mathbb{C}_u[X]$.  For each $n \geq 1$, we know that $f(p(x))\in W_1(p(x))$. By the definition of $\alpha_t$ and $\beta_t$, and Remark \ref{composition}, we have
\begin{equation*}\begin{aligned}
(\alpha_t\circ\beta_t(g\wox T))(x,y)
&=(\alpha_t(\beta(x)(g_t)\wox T))(x,y) \\
&=\theta_t^1(x)(\beta_{W_1(p(x))}(x)(g_t)\wox T(x,y))  \\
&= \theta_t^1(x)(g_t(X\wox 1+1\wox C_{W_1(p(x)),x})\wox T(x,y))
\end{aligned}\end{equation*}
where  $C_{W_1(p(x)),p(x)}$ is the Clifford algebra-valued function on $W_1(p(x))$ is defined to be  
$$C_{W_1(p(x)),p(x)}(v)=v-f(p(x))\in W_1(p(x)) \subset \text{Cliff}(E)$$ 
for all $v\in W_1(p(x))$.

Let $\gamma_t$ be the asymptotic morphism
$$\gamma_t:\mathcal{S}\wox C^*_u(X)\to\mathcal{S}\wox UC^*(X)$$
defined by
$$(\gamma_t(g\wox T))(x,y)=\theta_t^1(x)(\beta_{W_1(p(x))}(x)(g_t)\wox T(x,y)).$$
	
For every $t\geq 1, s\in[0,1]$ and $x\in X$, let $U_{x}(t)$ be the unitary operator acting on $\mathcal{H}$ induced by the translation $$v\mapsto v+t f(p(x))$$ on $V$. That is, $$(U_{x}(t)\xi)(v)=\xi(v+tf(p(x)))$$ for $\xi\in \mathcal{H}$. Write $\tilde{U}_{x}(t)$ as $1\hat{\otimes}1\hat{\otimes}U_{x}(t)\hat{\otimes}1$ acting on $L^2(\mathbb{R})\hat{\otimes}\ell^2(X)\hat{\otimes}\mathcal{H}\hat{\otimes}H_0$. We define a unitary operator  $U_{x,s}(t)$	acting on $L^2(\mathbb{R})\hat{\otimes}\ell^2(X)\hat{\otimes}\mathcal{H}\hat{\otimes}H_0\oplus L^2(\mathbb{R})\hat{\otimes}\ell^2(X)\hat{\otimes}\mathcal{H}\hat{\otimes}H_0$ by:
\begin{equation*}
U_{x,s}(t)=R(s)\left(\begin{matrix}
\tilde{U}_{x}(t) & 0 \\
0      & 1
\end{matrix}\right)R(s)^{-1},
\end{equation*}
where
\begin{equation*}R(s)=\left(\begin{matrix}
cos(\frac{\pi}{2}s)  & sin(\frac{\pi}{2}s)\\
-sin(\frac{\pi}{2}s) & cos(\frac{\pi}{2}s)
\end{matrix}\right).\end{equation*}
	
For each $s\in[0,1]$, we define an asymptotic morphism
$$\gamma_t(s):\mathcal{S}\wox C^*_u(X)\to\mathcal{S}\wox UC^*(X)\wox M_2(\IC)$$
by
\begin{equation*}
\gamma_t(s)(g\wox T)(x,y)=U_{x,s}(t)
\left(\begin{matrix}
\gamma_t(g\wox T)(x,y)     & 0 \\
0      & 0
\end{matrix}\right)U_{x,s}(t)^{-1}.
\end{equation*}
for each $g\in\mathcal{S}$, $T\in\IC_u[X]$, $t\geq 1$, $s\in[0,1]$, where the algebra of all complex 2-by-2 matrices $M_2(\IC)$ is endowed with the trivial grading.

Since $f(p(x))\subset W_1(p(x))$, we have that $$U_{x}(t)\theta_t^1(x)(\beta_{W_1(p(x))}(g_t)\hat{\otimes}T(x,y)) U_{x}(t)^{-1}=\theta_t^1(x)(\beta_{W_1(p(x))}(0)(g_t)\hat{\otimes}T(x,y)),$$
where $\beta_{W_1(p(x))}(0):\mathcal{S} \rightarrow \mathcal{A}(W_1(p(x)))$ is defined by:
$$(\beta_{W_1(p(x))}(0))(g)=g(X\hat{\otimes}1+1\hat{\otimes}C_{W_1(p(x))}).$$

Thus, by the proof of \cite[Proposition 4.2]{HKT1998}, we have that $\gamma_t(0)$ is asymtotic to an asymptotic morphism $\left(\begin{matrix}
\gamma_t' & 0 \\
0     & 0
\end{matrix}\right)$, where $\gamma_t'$ is  defined by
$$\gamma_t'(g\hat{\otimes}T)(x,y)\mapsto g_{t^2}(B_{0,t}(x))\hat{\otimes}T(x,y)$$ 
where $$B_{0,t}=t_0(D_0+C_0)+t_1(D_1+C_1)+\cdots+t_n(D_n+C_n)+\cdots$$
where $t_j=1+t^{-1}j$, $D_n$ and $C_n$ are respectively the Dirac operator and Clifford operator associated to $W_{n+1}(p(x))\ominus W_n(p(x))$. 

By \cite[Lemma 5.8]{HKT1998}, replacing $B_{0,t}$ with $s^{-1}B_{0,t}$ in the definition of $\gamma_t'$, for $s\in [0,1]$, we obtain a homotopy between  $\gamma_t'$ and the $*$-homomorphism
$$g\hat{\otimes}T\mapsto g\hat{\otimes}g(0)P\hat{\otimes}T$$
from $\mathcal S\wox C^*_u(X)$ to $\mathcal S\wox UC^*(X)$, where $P$ is the rank one projection onto the kernel of $B_{0,t}(x)$ and $P$ does not depend on $x$. Consequently, we have that $\alpha_t\circ\beta_t(g\hat{\otimes}T)$ is homotopic to $g\hat{\otimes}g(0)P\hat{\otimes}T$ for all $g\hat{\otimes}T\in\mathcal S\wox C^*_u(X)$. Moreover, the composition $(\alpha_t\circ \beta_t)_{t\in [1,\infty)}$ is representative of the composition of the asymptotic morphisms $(\alpha_t)_{t\in [1,\infty)}$ and $(\beta_t)_{t\in [1,\infty)}$ (see the details \cite[Appendix A]{JR2013} for a remark on compositions of asymptotic morphisms). Hence, by Remark \ref{Morita equivalence}, the composition
$E_* \circ\alpha_* \circ \beta_*$ are both identity maps.
\end{proof}

As for the Roe algebras, similarly, we have the following proposition.
\begin{Pro}\label{Roeiso}
{\it 
	The compositions
	$$\alpha_* \circ \beta_*: K_*(\mathcal{S}\hat{\otimes}C^*(X)) \rightarrow K_*(\mathcal{S}\hat{\otimes}C^*(X))$$ 
and its maximal counterpart
	$$\alpha_* \circ \beta_*: K_*(\mathcal{S}\hat{\otimes}C_{max}^*(X)) \rightarrow K_*(\mathcal{S}\hat{\otimes}C_{max}^*(X))$$
are identity maps.
}
\end{Pro}

Finally, we are ready to prove the main result of this paper.

\begin{proof}[Proof of Theorem \ref{main result}:]	
For the maximal and reduced uniform Roe algebras, 
consider the following commutative diagram:
$$\xymatrix{K_*(\mathcal{S}\hat{\otimes}C_{u,max}^*(X))\ar[r]^{\lambda_*}\ar[d]_{\beta_*} & K_*(\mathcal{S}\hat{\otimes}C_u^*(X))  \ar[d]^{\beta_*} \\
K_*(C_{u,max}^*(X,\mathcal{A}))\ar[r]^{\lambda_*}_{\cong}\ar[d]_{\alpha_*}&K_*(C_u^*(X,\mathcal{A}))\ar[d]^{\alpha_*} \\
K_*(\mathcal{S}\hat{\otimes}UC_{max}^{*}(X))\ar[r]^{\lambda_*}\ar[d]_{E_*}&K_*(\mathcal{S}\hat{\otimes}UC^*(X))\ar[d]^{E_*} \\
K_*(\mathcal{S}\hat{\otimes}C_{u,max}^*(X))\ar[r]^{\lambda_*}& K_*(\mathcal{S}\hat{\otimes}C_u^*(X))}$$
By Proposition \ref{composition iso} the two vertical compositions are isomorphisms. By Theorem \ref{max-red} the second horizontal homomorphism is an isomorphism. Hence, Theorem \ref{main result} follows from a diagram chasing argument.

As for the maximal and reduced Roe algebras, we consider the following commutative diagram: 
	$$\xymatrix{K_*(\mathcal{S}\hat{\otimes}C_{max}^*(X))\ar[r]^{\lambda_*}\ar[d]_{\beta_*} & K_*(\mathcal{S}\hat{\otimes}C^*(X))  \ar[d]^{\beta_*} \\
		K_*(C_{max}^*(X,\mathcal{A}\hat{\otimes}\mathcal{K}))\ar[r]^{\lambda_*}_{\cong}\ar[d]_{\alpha_*}&K_*(C^*(X,\mathcal{A}\hat{\otimes}\mathcal{K}))\ar[d]^{\alpha_*} \\
		K_*(\mathcal{S}\hat{\otimes}C_{max}^*(X))\ar[r]^{\lambda_*}& K_*(\mathcal{S}\hat{\otimes}C^*(X))}$$
By Proposition \ref{Roeiso} the two vertical compositions are isomorphisms. By Theorem \ref{max-red} the second horizontal homomorphism is an isomorphism. Via a diagram chasing argument, $\lambda_*$ is an isomorphism as desired.
	
\end{proof}

\section*{Acknowledegments}
We would like to thank Jintao Deng and Guoliang Yu for valuable suggestions and helpful discussions.

\vskip 1cm
\begin{itemize}
\item[] Liang Guo\\
Research Center for Operator Algebras, School of Mathematical Sciences, East China Normal University, Shanghai, 200241, P. R. China.\quad
E-mail: 52205500015@stu.ecnu.edu.cn

\item[] Zheng Luo \\
Research Center for Operator Algebras, School of Mathematical Sciences, East China Normal University, Shanghai, 200241, P. R. China.\quad
E-mail: 52195500005@stu.ecnu.edu.cn

\item[] Qin Wang \\
Research Center for Operator Algebras,  and Shanghai Key Laboratory of Pure Mathematics and Mathematical Practice, School of Mathematical Sciences, East China Normal University, Shanghai, 200241, P. R. China. \quad
E-mail: qwang@math.ecnu.edu.cn

\item[] Yazhou Zhang \\
Research Center for Operator Algebras, School of Mathematical Sciences, East China Normal University, Shanghai, 200241, P. R. China.\quad
E-mail: 52185500010@stu.ecnu.edu.cn

\end{itemize}

\end{document}